\documentclass{article}%
\usepackage{amsmath}%
\usepackage{amsfonts}%
\usepackage{amssymb}%
\usepackage{graphicx}
\usepackage{enumerate}
\newtheorem{theorem}{Theorem}

\newtheorem{proposition}[theorem]{Proposition}

\newenvironment{proof}[1][Proof]{\textbf{#1.} }{\ \rule{0.5em}{0.5em}}

\newcommand{\ENT}[1]{ \left\lfloor #1 \right\rfloor}
\newcommand{\Efrac}[2]{\ENT{\frac{#1}{#2}}}
\newcommand{\FRA}[1]{\left\{ #1 \right\}}
\newcommand{\SUMS}[3]{\mathbf{S}_{#1}\left(#2,#3\right)}
\newcommand{\SUMW}[2]{\mathbf{W}_{#1}\left(#2\right)}
\newcommand{\SUMF}[2]{\mathbf{F}_{#1}\left(#2\right)}
\newcommand{\doublepariii}[1]{\left(\!\!\!\left(#1\right)\!\!\!\right)}

\newcommand{\doublepari}[1]{\left(\!\left(#1\right)\!\right)}

\begin{document}

\title{On a family of sums of powers of the floor function and their links with generalized Dedekind sums}
\author{Steven Brown}
\date{July 15, 2025}
\maketitle

\begin{abstract}
In this paper we are concerned with a family of sums involving the floor function. With $r$ a non negative integer and $n$ and $m$ positive integers we consider the sums 
\begin{equation*}
\SUMS{r}{n}{m}:=\sum_{k=1}^{n-1}{\Efrac{km}{n}}^r
\end{equation*}
While a formula for $\mathbf{S}_1$ is well known, we  provide closed-form formulas for $\mathbf{S}_2$ and $\mathbf{S}_3$ as well as the reciprocity laws they satisfy. Additionally, one can find a closed-form formula for the classical Dedekind sum using the Euclidean algorithm. Finally, we provide a general formula for $\mathbf{S}_r$ showing its dependency on generalized Dedekind sums.
\end{abstract}

\section{Introduction}\label{sec:Introduction}

The analysis of the sums $\SUMS{r}{n}{m}$ defined below (see \ref{eq:def_Srnm}) was motivated due to their links with Dedekind sums (see proposition \ref{prop:Srba_rel_Dedekind_sums}) which have applications in many areas of mathematics. A broad overview of the Dedekind sums and their applications can be found in the introduction of \cite{rademacher1972dedekind}, a reference monograph on the topic by Hans Rademacher and Emil Grosswald.

\subsection{Notations} \label{sec:Notations}

For any real number $x$ we denote by $\ENT{x}$ the floor function defined as the greatest integer less than or equal to $x$. For an integer $a$ and a positive integer $b$, we use $a \mod b$ to mean the remainder of $a$ when divided by $b$.

The sums of interest are noted as follows:
\begin{equation} \label{eq:def_Srnm}
\SUMS{r}{n}{m}:=\sum_{k=1}^{n-1}{\Efrac{km}{n}}^r
\end{equation}

We also use the follwing notation for the Faulhaber sums (see \cite{schumacher2016extended}).
\begin{equation}
\SUMF{r}{n}:=\sum_{i=0}^n i^r
\label{eq:def_Frn}
\end{equation}

For positive coprime integers $a$ and $b$ we use the notation $s(b,a)$ as in Rademacher's book \cite{rademacher1972dedekind} to denote the classical Dedekind sum.

\begin{equation}
s(b,a):=\sum_{k=1}^a \doublepariii{\frac{kb}{a}}\doublepariii{\frac{k}{a}}
\label{eq:def_Dedekind_sum}
\end{equation}

with the symbol $\doublepari{x}$ defined by

\begin{equation}
\doublepari{x} := \left\{
    \begin{array}{ll}
        \{x\}-\frac{1}{2} & \mbox{if  } x\notin\mathbb{Z} \\
				0                 & \mbox{if  } x\in\mathbb{Z}
    \end{array}
\right.
\end{equation}

We use as well a definition of generalized Dedekind sums suggested  by Don Zagier in the article \cite{zagier1973higher} (equation 40 page 157). In this definition, $b$ and $n$ are positive integers and the coefficients $a_i$ for $i$ from 1 to $n$ are positive integers coprime with $b$.

\begin{equation}
\delta(b;a_1,\ldots ,a_n) := 2^nb^{n-1}\sum_{k=1}^{b-1}\prod_{i=1}^n \doublepariii{\frac{ka_i}{b}}
\label{eq:def_generalized_Dedekind_sum}
\end{equation}

\subsection{General considerations}\label{sec:General_considerations}

\begin{proposition}\label{prop:rel_Srba_Siba} Let $a$, $b$ and $r$ be positive coprime integers. If $a$ and $b$ are coprime then
\begin{equation}
\SUMS{r}{b}{a}+\sum_{i=0}^r(-1)^{i+1}\binom{r}{i}(a-1)^{r-i}\SUMS{i}{b}{a}=0
\label{eq:rel_Srba_Siba}
\end{equation} 
\end{proposition}

\begin{proof}
Let $k$ be an integer satisfying $0<k<b$. Let $u_k:=\ENT{\frac{ka}{b}}$ and $v_k:=\ENT{\frac{-ka}{b}}$. The ratio $\frac{ka}{b}$ is not in $\mathbb{Z}$ and therefore, as a property of the floor function\footnote{$\forall x \in \mathbb{R}, \quad \left\lfloor x\right\rfloor + \left\lfloor -x \right\rfloor = 
\left\{
    \begin{array}{ll}
				0  & \mbox{if  } x\in\mathbb{Z} \\
        -1 & \mbox{if  } x\notin\mathbb{Z}
    \end{array}
\right.$}, we have $u_k+v_k=-1$

\begin{align*}
\ENT{\frac{ka}{b}}^r-\ENT{\frac{(b-k)a}{b}}^r
&= {u_k}^r-(a+v_k)^r \\
&= {u_k}^r-(a-1-u_k)^r \\
&= {u_k}^r-\sum_{i=0}^r\binom{r}{i}(-u_k)^i(a-1)^{r-i} \\
&= {u_k}^r+\sum_{i=0}^r(-1)^{i+1}\binom{r}{i}(a-1)^{r-i}{u_k}^i \\
\end{align*}

We can take the sum of the last equation for $k$ from 1 to $b-1$. The left hand side sums to zero as the difference of two equal sums (the two sums index are in reverse order). On the right hand side we recognize the sums $\SUMS{i}{b}{a}$ for $i$ from 0 to $r$.
\end{proof}

\begin{proposition}\label{prop:rel_Srnm_Siba} Let $m$, $n$ and $r$ be positive integers. let $d=\gcd(n,m)$ such that there exist two positive coprime integers $b$ and $a$ such that $n=db$ and $m=da$. We have
\begin{equation}
\SUMS{r}{n}{m} = a^r\SUMF{r}{d-1}+\sum_{k=0}^{r}\binom{r}{k}a^k\SUMF{k}{d-1}\SUMS{r-k}{b}{a}
\label{eq:rel_Srnm_Siba}
\end{equation}
\end{proposition}

\begin{proof}
\begin{align*}
\SUMS{r}{n}{m}
&= \sum_{k=0}^{n-1}\ENT{\frac{ka}{b}}^r \\
&= \sum_{i=0}^{d-1}\sum_{j=0}^{b-1}\ENT{\frac{(ib+j)a}{b}}^r \\
&= \sum_{i=0}^{d-1}\sum_{j=0}^{b-1}\sum_{k=0}^r\binom{r}{k}i^ka^k\ENT{\frac{ja}{b}}^{r-k} 
\end{align*}
 
The case $k=r$ needs attention since $\sum_{j=0}^{b-1}\ENT{\frac{ja}{b}}^{r-r}=1+\SUMS{0}{b}{a}$

\begin{align*}
\SUMS{r}{n}{m} 
&= a^r\SUMF{r}{d-1}(1+\SUMS{0}{b}{a})+\sum_{k=0}^{r-1}\binom{r}{k}a^k\SUMF{k}{d-1}\SUMS{r-k}{b}{a} \\
&= a^r\SUMF{r}{d-1}+\sum_{k=0}^{r}\binom{r}{k}a^k\SUMF{k}{d-1}\SUMS{r-k}{b}{a} \\
\end{align*}

\end{proof}

Equation \ref{eq:rel_Srnm_Siba} shows that in general (whether $m$ and $n$ are coprime or not) the formula of $\SUMS{r}{n}{m}$ only depend on $\SUMS{i}{b}{a}$ for $0\leq i \leq r$ and some known Faulhaber sums. Therefore it is enough to focus on studying $\SUMS{r}{b}{a}$ with $b$ and $a$ positive coprime integers.

\section{Formulas for $\mathbf{S}_i$ for $0\leq i\leq 3$}

\subsection{A formula for $\SUMS{1}{n}{m}$}\label{sec:S1nm}

A formula for $\SUMS{1}{n}{m}$ is provided and proved in \cite{polezzi1997Geometrical} and also in \cite{graham1994concrete} page 94.
If $m$ and $n$ are positive integers and if $d=\gcd(m,n)$ then

\begin{equation}
\SUMS{1}{n}{m} = \frac{(m-1)(n-1)}{2}+\frac{d-1}{2} 
\label{eq:formula_S1nm}
\end{equation}

We give here a first\footnote{a second proof is given in section \ref{sec:Srba_fn_gen_DedekindSums}} alternative proof of equation \ref{eq:formula_S1nm}. Let $a$ and $b$ be the positive coprime integers defined by $m=da$ and $n=db$. From proposition \ref{prop:rel_Srba_Siba} with $r=1$ we have

\begin{equation*}
2\SUMS{1}{b}{a} - (a-1)\SUMS{0}{b}{a} = 0
\end{equation*}

It is clear that

\begin{equation}
\SUMS{0}{b}{a}=b-1
\label{eq:formula_S0ba}
\end{equation}

and therefore

\begin{equation}
\SUMS{1}{b}{a} = \frac{(a-1)(b-1)}{2}
\label{eq:formula_S1ba}
\end{equation}

We write now equation \ref{eq:rel_Srnm_Siba} from proposition \ref{prop:rel_Srba_Siba} with $r=1$

\begin{equation*}
\SUMS{1}{n}{m} = a\SUMF{1}{d-1}+\SUMF{0}{d-1}\SUMS{1}{b}{a}+a\SUMF{1}{d-1}\SUMS{0}{b}{a}
\end{equation*}

This equation gives equation \ref{eq:formula_S1nm} knowing that $\SUMF{0}{d-1}=d$ and $\SUMF{1}{d-1}=\frac{(d-1)d}{2}$.

\subsection{A formula for $\SUMS{2}{n}{m}$} 

In this section we carry out a direct calculation of $\SUMS{2}{n}{m}$ and establish an equation involving another sum of interest that will be studied separately. We have the positive integers $m$, $n$, $a$, $b$, $d$ and $k$ such that $m=da$, $n=db$ and $a$ and $b$ are coprime. We write now equation \ref{eq:rel_Srnm_Siba} from proposition \ref{prop:rel_Srnm_Siba} with $r=2$.

\begin{multline*}
\SUMS{2}{n}{m} = a^2\SUMF{2}{d-1} + \\ \SUMF{0}{d-1}\SUMS{2}{b}{a} + 2a\SUMF{1}{d-1}\SUMS{1}{b}{a} + a^2\SUMF{2}{d-1}\SUMS{0}{b}{a}
\end{multline*}

Since all is known apart from $\SUMS{2}{b}{a}$, we have

\begin{equation*} \label{eq:rel_S2np_S2ba}
\SUMS{2}{n}{m} = \frac{(d-1)m}{6}\left((2d-1)ab+3(a-1)(b-1)\right)+d\SUMS{2}{b}{a}
\end{equation*}

Since we have

\begin{equation} \label{eq:fond}
\ENT{\frac{ka}{b}} = \frac{ka}{b} - \FRA{\frac{ka}{b}}
\end{equation}

Summing the square of equation \ref{eq:fond} for $k$ from $1$ to $b-1$ leads to

\begin{align*}
\SUMS{2}{b}{a} 
&= \sum_{k=1}^{b-1}\left({\frac{ka}{b}}\right)^2 + \sum_{k=1}^{b-1}\FRA{\frac{ka}{b}}^2 - 2 \sum_{k=1}^{b-1}\frac{ka}{b}\FRA{\frac{ka}{b}} \\
&= \frac{1+a^2}{b^2}\SUMF{2}{b-1}-2\frac{a}{b^2}\sum_{k=1}^{b-1}k(ka\mod b)
\end{align*}

Let us define the function $\SUMW{n}{a,b}$ by the following sum

\begin{equation} \label{eq:Wnab}
\SUMW{n}{a,b} :=\sum_{k=1}^{n-1}(ak\mod n)(bk\mod n)
\end{equation}

We also give a sense to this function when at least one of its arguments is equal to one through this definition

\begin{equation} \label{eq:Wna}
\SUMW{n}{a}:=\SUMW{n}{a,1}
\end{equation}

With this definition we get to

\begin{equation} \label{eq:rel_S2ba_Wba}
\SUMS{2}{b}{a} = \frac{(1+a^2)(b-1)(2b-1)}{6b}-2\frac{a}{b^2}\SUMW{b}{a}
\end{equation}

And therefore

\begin{equation} \label{eq:rel_S2nm_Wba}
\SUMS{2}{n}{m} = \frac{d}{6b}\left((b-1)(2b-1)+a^2(n-1)(2n-1)\right)-\frac{m}{2}(d-1)(b-1)-2\frac{m}{b^2}\SUMW{b}{a}
\end{equation}

This formula together with equation \ref{eq:Wab_fn_ui} from section \ref{sec:Wab_Euclidean_algorithm} provides a closed-form formula for $\SUMS{2}{n}{m}$.

\subsection{A formula for $\SUMS{3}{n}{m}$}

Let us write equation \ref{eq:rel_Srba_Siba} from proposition \ref{prop:rel_Srba_Siba} with $r=3$

\begin{equation*}
2\SUMS{3}{b}{a}-(a-1)^3\SUMS{0}{b}{a}+3(a-1)^2\SUMS{1}{b}{a}-3(a-1)\SUMS{2}{b}{a} = 0
\end{equation*}

By means of equations \ref{eq:formula_S0ba}, \ref{eq:formula_S1ba} and \ref{eq:rel_S2ba_Wba} the above equation gives:

\begin{equation}
\SUMS{3}{b}{a}=\frac{1}{4b}(b-1)(a-1)\left((b-1)(1+a^2)+2ab\right)-\frac{3}{b^2}a(a-1)\SUMW{b}{a}
\label{eq:rel_S3ba_Wba}
\end{equation}

We now write equation \ref{eq:rel_Srnm_Siba} from proposition \ref{prop:rel_Srnm_Siba} with $r=3$. There is

\begin{multline*}
\SUMS{3}{n}{m} = a^3\SUMF{3}{d-1}+\SUMF{0}{d-1}\SUMS{3}{b}{a}+\\
3a\SUMF{1}{d-1}\SUMS{2}{b}{a}+3a^2\SUMF{2}{d-1}\SUMS{1}{b}{a}+a^3\SUMF{3}{d-1}\SUMS{0}{b}{a}
\end{multline*}

Using known formulas for Faulhaber sums as well as equations \ref{eq:formula_S0ba}, \ref{eq:formula_S1ba}, \ref{eq:rel_S2ba_Wba} and \ref{eq:rel_S3ba_Wba} we get to

\begin{multline}\label{eq:rel_S3nm_Wba}
\SUMS{3}{n}{m} = \frac{1}{4}(d-1)am\left((d-1)bm+(2d-1)(a-1)(b-1)\right)\\
+\frac{1}{4b}d(1+a^2)(b-1)\left((d-1)a(2b-1)+(b-1)(a-1)\right)\\
+\frac{1}{2}m(b-1)(a-1)-\frac{3}{b^2}m(m-1)\SUMW{b}{a}
\end{multline}

This formula together with equation \ref{eq:Wab_fn_ui} from section \ref{sec:Wab_Euclidean_algorithm} provides a closed-form formula for $\SUMS{3}{n}{m}$.

\section{Analysis of $\mathbf{W}$}

The objective of this section is to give a closed-form formula for $\SUMW{n}{m}$ in order to finalize the calculation of $\SUMS{2}{n}{m}$ in equation \ref{eq:rel_S2nm_Wba} and of $\SUMS{3}{n}{m}$ in equation \ref{eq:rel_S3nm_Wba}. The analysis of $\mathbf{W}$ provides an elementary proof\footnote{Although it is not fundamentally a new proof, one can see it in section \ref{sec:proof_Dedekind_reciprocity_law}} of the simplest form of Dedekind's reciprocity law.

\subsection{Basic properties of $\mathbf{W}$}

In the previous section we introduced the function $\mathbf{W}$ in equations \ref{eq:Wnab} and \ref{eq:Wna}. The objective of this section is to provide some of its properties. 

\begin{proposition}\label{prop:basic_properties_W}
Let $n$, $a$, $b$ and $c$ be any positive integers we have the following
\begin{center}
\begin{enumerate}[i]
	\item $\SUMW{n}{a,b}=\SUMW{n}{a\mod n,b\mod n}$
	\item If $\gcd(c,n)=1$ we have $\SUMW{n}{ac,bc}=\SUMW{n}{a,b}$
	\item If $ab\mod n=1$ we have $\SUMW{n}{a}=\SUMW{n}{b}$
	\item $\SUMW{n}{a}+\SUMW{n}{n-a}=\frac{1}{2}n^2(n-1)$
\end{enumerate}
\end{center}
\end{proposition}
\begin{proof} 
\begin{enumerate}[i]
  \item comes from $ak\mod n=((a\mod n)k)\mod n$. 
	\item Whenever positive integer $c$ is coprime with $n$, the application $x\mapsto cx\mod n$ is a bijection of $\left\{1,\ldots,n-1\right\}$. In that case we have $\SUMW{n}{ac,bc}=\SUMW{n}{a,b}$.
	\item \begin{align*}
\SUMW{n}{a} 
&= \sum_{k=1}^{n-1}(bk\mod n)(a(bk\mod n)\mod n) \\
&= \sum_{k=1}^{n-1}(bk\mod n)((ab\mod n)k\mod n) \\
&= \SUMW{n}{b}
\end{align*}
For the second equality we use the fact that $ab\mod n=1$ implies that $\gcd(b,n)=1$ and therefore $k\mapsto bk\mod n$ is a bijection of $A_n$.
	\item \begin{equation*}
\SUMW{n}{a}+\SUMW{n}{n-a} = \sum_{k=1}^{n-1} k \left\{(ak \mod n)+(-ak \mod n) \right\} 
\end{equation*}
\end{enumerate}
\end{proof}

\begin{proposition}\label{prop:rel_Wnm_Wba}
Let $d=\gcd(m,n)$, we can write $m=da$ and $n=db$ with $a$ and $b$ coprime. We have the following equation
\begin{equation}
\SUMW{n}{m} = d^2 \SUMW{b}{a} + \frac{1}{4}n^2(d-1)(b-1)
\label{eq:rel_Wnm_Wba}
\end{equation}
 
\end{proposition}

\begin{proof}
\begin{align*}
\SUMW{n}{m}
&= \sum_{k=1}^{db-1}k(dak \mod db) \\
&= d \sum_{j=0}^{d-1} \sum_{k=0}^{b-1}(k+jb)(a(k+jb) \mod b) \\
&= d^2 W_{b}(a) + db \sum_{j=0}^{d-1} j\sum_{k=0}^{b-1}(ak \mod b) \\
&= d^2 W_{b}(a) + \frac{1}{4}n^2(d-1)(b-1) \\
\end{align*}
\end{proof}

Note that when $m$ and $n$ are coprime, then $d=1$ and the equation is obviously satisfied. We can now focus on calculating $\SUMW{b}{a}$ when $b$ and $a$ are coprime.

\subsection{Calculation of $\SUMW{b}{a}$ when $a$ and $b$ are coprime}

Given that from property (ii) $\SUMW{b}{a}=\SUMW{b}{a\mod b}$ and that $\gcd(a,b)=1$ implies $\gcd(a\mod b,b)=1$ we can work under the assumption that $0<a<b$ even if it means considering $a\mod b$ instead of $a$. According to definition \ref{eq:Wnab}

\begin{equation*}
\SUMW{b}{a} = \sum_{k=1}^{b-1}k(ak \mod b)
\end{equation*}

Note that the term $k(ak \mod b)$ inside the sum is zero for $k=b$. In particular we can write $\SUMW{b}{a}$ in a slightly different way

\begin{align*}
\SUMW{b}{a} = \sum_{j=0}^{a-1}\sum_{k=\Efrac{jb}{a}+1}^{\Efrac{(j+1)b}{a}}k(ak \mod b) 
\end{align*}

The integer variable $k$ of the inner sum satisfies

\begin{equation*}
\frac{jb}{a} < \Efrac{jb}{a}+1 \leq k \leq \Efrac{(j+1)b}{a} \leq \frac{(j+1)b}{a}
\end{equation*}

Which implies

\begin{equation*}
0 < ka-jb \leq b
\end{equation*}

It should be noted that the right hand side inequality is always a strict inequality apart from the case when $j=a-1$ and $k=b$. Indeed, when $0\leq j<a-1$, the ratio $\frac{(j+1)b}{a}$ is never an integer. If that was the case, knowing that $a$ and $b$ are coprime, the Gauss lemma would imply that $a$ divides $j+1$ which is not possible since $0< j+1 < a$. That means that apart from the case $j=a-1$ and $k=b$ we have

\begin{equation*}
ka\mod b = ka-jb 
\end{equation*}

Now we can write

\begin{equation*}
\SUMW{b}{a} = \left( \sum_{j=0}^{a-1}\sum_{k=\Efrac{jb}{a}+1}^{\Efrac{(j+1)b}{a}}k(ka-jb) \right) - b^2
\end{equation*}

Note that when $j=a-1$ and $k=b$, the expressions $k(ka-jb)=b^2$ and $k(ka\mod b)=0$ are not equal, the reason why we need to substract $b^2$.

The first part of the sum is easily simplified

\begin{align*}
\sum_{j=0}^{a-1}\sum_{k=\Efrac{jb}{a}+1}^{\Efrac{(j+1)b}{a}}ak^2 
&= a \sum_{k=1}^{b}k^2 \\
&= a \frac{b(b+1)(2b+1)}{6}
\end{align*}

We are now left with the calculation of the second term of the sum 

\begin{equation*}
A:=-b \sum_{j=0}^{a-1} j \sum_{k=\Efrac{jb}{a}+1}^{\Efrac{(j+1)b}{a}}k
\end{equation*}

We have
\begin{align*}
A
&= -\frac{b}{2}\sum_{j=0}^{a-1}j\left( \Efrac{(j+1)b}{a}\left(\Efrac{(j+1)b}{a}+1\right)-\Efrac{jb}{a}\left(\Efrac{jb}{a}+1\right) \right) \\
&= -\frac{b}{2}\sum_{j=1}^{a}(j-1) \Efrac{jb}{a}\left(\Efrac{jb}{a}+1\right)+\frac{b}{2}\sum_{j=0}^{a-1}j\Efrac{jb}{a}\left(\Efrac{jb}{a}+1\right) \\
&= \frac{b}{2}\SUMS{2}{a}{b}+\frac{b}{2}\SUMS{1}{a}{b} - \frac{b}{2}(a-1)b(b+1)
\end{align*}

Using equation \ref{eq:formula_S1ba} and after some simplifications we get to

\begin{equation}\label{eq:rel_Wba_S2ab}
\SUMW{b}{a} = \frac{b}{2}\SUMS{2}{a}{b} + \frac{b}{12}(b-1)(2b-1)(3-a)
\end{equation}

\subsection{Formula for $\SUMW{a}{b}$ using the Euclidean algorithm}\label{sec:Wab_Euclidean_algorithm}

The purpose of this section is to provide a closed-form formula for $\SUMW{a}{b}$ for two positive coprime integers $a$ and $b$ with $a<b$. 
From equation \ref{eq:reciprocity_W} and using property (ii) we can write

\begin{equation}
\SUMW{a}{b} = f(a,b)-\left(\frac{a}{b}\right)^2\SUMW{b}{a\mod b}
\label{eq:induction_W}
\end{equation}

With $f$ being the following function
\begin{equation}
f(x,y) := \frac{x}{12y}\left((1+x^2)(1+y^2)-xy(x-3)(y-3)\right)
\label{eq:def_f}
\end{equation}

Let $(u_n)_{n\in \mathbb{N}}$ be the sequence defined by the first two terms, $u_0=a$, $u_1=b$ and the following induction equation $u_{i+2}=u_i \mod u_{i+1}$ for $i\geq 0$. This sequence is the sequence of remainders of Euclid's algorithm (see \cite{damphousseopuscules}). We know that $(u_n)_{n\in \mathbb{N}}$ is strictly decreasing until it reaches $u_N=1=\gcd(a,b)$ for a specific index $N\geq 1$. Then for any $i>N$ we have $u_i=0$. For $i$ from 0 to $N-1$ we have $\gcd(u_i,u_{i+1})=1$ and we can write $N$ times equation \ref{eq:induction_W} for $\SUMW{u_i}{u_{i+1}}$. Compounding these $N$ equations leads to
\begin{equation*}
\SUMW{u_0}{u_1}=\left( \sum_{k=0}^{N-1}{(-1)}^k{\left(\frac{u_0}{u_k}\right)}^2f(u_k,u_{k+1})\right)+{(-1)}^N{\left(\frac{u_0}{u_N}\right)}^2\SUMW{u_N}{u_{N+1}}
\end{equation*}
Given that $u_{N+1}=0$, we have $\SUMW{u_N}{u_{N+1}}=0$ and we are left with
\begin{equation}
\SUMW{a}{b}=\frac{a^2}{12}\sum_{k=0}^{N-1}{(-1)}^k\left(\frac{(1+{u_k}^2)(1+{u_{k+1}}^2)}{u_k u_{k+1}}-(u_k-3)(u_{k+1}-3)\right)
\label{eq:Wab_fn_ui}
\end{equation}

As a consequence, we have closed-form formulas\footnote{Don Zagier in \cite{zagier1973higher} page 166 had already noticed that the classical Dedekind sum was fully determined from their properties and the use of the Euclidean algorithm.} for $\SUMS{2}{b}{a}$, $\SUMS{2}{n}{m}$, $\SUMS{3}{b}{a}$, $\SUMS{3}{n}{m}$, $\SUMW{n}{m}$ and the classical Dedekind sum $s(b,a)$ respectively from equations \ref{eq:rel_S2ba_Wba}, \ref{eq:rel_S2nm_Wba}, \ref{eq:rel_S3ba_Wba}, \ref{eq:rel_S3nm_Wba}, \ref{eq:rel_Wnm_Wba} and \ref{eq:rel_Wab_sba}. 

\section{Reciprocity laws}

With positive and coprime integers $a$ and $b$, the consideration of equations \ref{eq:rel_Wba_S2ab} and \ref{eq:rel_S2ba_Wba} yields easily to the following symetrical equations that could be considered as reciprocity laws:

\begin{theorem}[Reciprocity law for $\mathbf{S}_2$]\label{thm:reciprocity_S2} If $a$ and $b$ are positive coprime integers then
\begin{equation}
a \SUMS{2}{a}{b} + b \SUMS{2}{b}{a} = \frac{1}{6}(a-1)(2a-1)(b-1)(2b-1)
\label{eq:reciprocity_S2}
\end{equation}
\end{theorem}

\begin{proof} In equation \ref{eq:rel_S2ba_Wba} we replace $\SUMW{b}{a}$ by its expression from equation \ref{eq:rel_Wba_S2ab}.
\end{proof}

\begin{theorem}[Reciprocity law for $\mathbf{W}$]\label{thm:reciprocity_W} If $a$ and $b$ are positive coprime integers then
\begin{equation}
a^2\SUMW{b}{a}+b^2\SUMW{a}{b}=\frac{ab}{12}\left((1+a^2)(1+b^2)-ab(a-3)(b-3)\right)
\label{eq:reciprocity_W}
\end{equation}
\end{theorem}

\begin{proof} In equation \ref{eq:rel_S2ba_Wba} we swap $a$ and $b$ and inject the expression of $\SUMS{2}{a}{b}$ in equation \ref{eq:rel_Wba_S2ab}.
\end{proof}

\begin{theorem}[Reciprocity law for $\mathbf{S}_3$]\label{thm:reciprocity_S3} If $a$ and $b$ are positive coprime integers then
\begin{equation}
a(a-1) \SUMS{3}{a}{b} + b(b-1) \SUMS{3}{b}{a} = \frac{1}{4}(a-1)^2(b-1)^2\left((a-1)(b-1)+ab\right)
\label{eq:reciprocity_S3}
\end{equation}
\end{theorem}

\begin{proof} In equation \ref{eq:rel_Srba_Siba} for $r=3$, we replace $\mathbf{S}_0$ and $\mathbf{S}_1$ according to their formulas in equations \ref{eq:formula_S0ba} and \ref{eq:formula_S1ba} and get an equation between $\mathbf{S}_2$ and $\mathbf{S}_3$. With this equation and the reciprocity law for $\mathbf{S}_2$ in \ref{eq:reciprocity_S2} we easily get equation \ref{eq:reciprocity_S3}.
\end{proof}

\subsection{A proof of Dedekind's reciprocity law}\label{sec:proof_Dedekind_reciprocity_law}

The proof of Dedekind's reciprocity law\footnote{Not to be mistaken with the Quadratic reciprocity law} that we give here is in essence the same as the one given in \cite{zagier1973higher} page 153 although it is presented differently.
In the definition equation \ref{eq:def_Dedekind_sum} of the classical Dedekind sum, the summand for $k=a$ is equal to 0. For $0<k<a$ both $\frac{kb}{a}$ and $\frac{k}{a}$ are not in $\mathbb{Z}$

\begin{align*}
s(b,a)
&= \sum_{k=1}^{a-1} \left(\frac{kb\mod a}{a}-\frac{1}{2}\right)\left(\frac{k}{a}-\frac{1}{2}\right)\\
&= \frac{1}{a^2}\SUMW{a}{b}-\frac{1}{4}(a-1)
\end{align*}

That is

\begin{equation}
\SUMW{a}{b}=a^2 \left( s(b,a)+ \frac{a-1}{4} \right)
\label{eq:rel_Wab_sba}
\end{equation}

This reciprocity law satisfied by the classical Dedekind sum results from the reciprocity law satisfied by $\mathbf{W}$ (equation \ref{eq:reciprocity_W}) and the relation between $\mathbf{W}$ and the classical Dedekind sum (equation \ref{eq:rel_Wab_sba}).
Combining these two equations yields

\begin{multline*}
a^2b^2 \left( s(b,a)+ \frac{a-1}{4} \right) + a^2b^2 \left( s(a,b)+ \frac{b-1}{4} \right) = \\ \frac{ab}{12}\left((1+a^2)(1+b^2)-ab(a-3)(b-3)\right)
\end{multline*}

From where we get the reciprocity law for Dedekind sums.

\begin{equation*}
s(b,a)+s(a,b) = -\frac{1}{4}+\frac{1}{12}\left(\frac{a}{b}+\frac{1}{ab}+\frac{b}{a}\right)
\end{equation*}

\subsection{A formula for the classical Dedekind sum}

From equation \ref{eq:rel_Wab_sba} we have

\begin{equation*}
s(b,a) = \frac{\SUMW{a}{b}}{a^2}-\frac{a-1}{4}
\end{equation*}

This equation together with equation \ref{eq:Wab_fn_ui} gives a closed-form formula for $s(b,a)$ as a function of the remainders obtained with the Euclidean algorithm applied to $u_0=b$ and $u_1=a$ (see section \ref{sec:Wab_Euclidean_algorithm}).

\section{Expression of $\SUMS{r}{b}{a}$ as a function of generalized Dedekind sums}\label{sec:Srba_fn_gen_DedekindSums}

\begin{proposition}\label{prop:Srba_rel_Dedekind_sums} for positive coprime integers $a$ and $b$ and for positve integer $r$ we have the following expression for $\SUMS{r}{b}{a}$
\begin{equation}
\SUMS{r}{b}{a} = \frac{b}{2^r}\sum_{\substack{u+v+w=r\\u,v,w\geq 0}}\binom{r}{u,v,w}\left(\frac{a}{b}\right)^u\left(\frac{-1}{b}\right)^v
(a-1)^w \delta(b;\underbrace{1,\ldots,1}_{u\textrm{ times}},\underbrace{a,\ldots ,a}_{v\textrm{ times}})
\label{eq:Srba_rel_Dedekind_sums}
\end{equation}
\end{proposition}

\begin{proof} We transform $\SUMS{r}{b}{a}$ using the trinomial expansion and recognize generalized Dedekind sums (equation \ref{eq:def_generalized_Dedekind_sum}) in the expression.
\begin{align*}
\SUMS{r}{b}{a} 
&= \sum_{k=1}^{b-1}\left(\frac{ka}{b}-\frac{a}{2}-\left(\left\{\frac{ka}{b}\right\}-\frac{1}{2}\right)+\frac{a-1}{2}\right)^r \\
&= \sum_{k=1}^{b-1} \sum_{\substack{u+v+w=r\\u,v,w\geq 0}}\binom{r}{u,v,w}a^u\left(\frac{k}{b}-\frac{1}{2}\right)^u(-1)^v\left(\left\{\frac{ka}{b}\right\}-\frac{1}{2}\right)^v \left(\frac{a-1}{2}\right)^w \\
&=  \sum_{\substack{u+v+w=r\\u,v,w\geq 0}}\binom{r}{u,v,w}a^u(-1)^v \left(\frac{a-1}{2}\right)^w \frac{2^{u+v}b^{u+v-1}}{2^{u+v}b^{u+v-1}} \sum_{k=1}^{b-1} \doublepariii{\frac{k}{b}}^u\doublepariii{\frac{ka}{b}}^v
\end{align*}
\end{proof}

The formula gives directly the results \ref{eq:formula_S1ba}, \ref{eq:rel_S2ba_Wba} and \ref{eq:rel_S3ba_Wba} given previously. We only give its application to get the formula for $\SUMS{1}{b}{a}$ as an example. We know that the generalized Dedekind sum \ref{eq:def_generalized_Dedekind_sum} is zero when $n$ is odd. Therefore the only possibility is $(u,v,w)=(0,0,1)$ then equation \ref{eq:Srba_rel_Dedekind_sums} becomes

\begin{equation*}
\SUMS{1}{b}{a} = \frac{b}{2} (a-1) \delta(b;\emptyset)
\end{equation*}

But

\begin{align*}
\delta(b;\emptyset)
&= 2^0b^{-1}\sum_{k=1}^{b-1}\prod_{a\in \emptyset} \doublepariii{\frac{ka}{b}} \\
&= \frac{(b-1)}{b}
\end{align*}

Hence equation \ref{eq:Srba_rel_Dedekind_sums} gives another proof of

\begin{equation*}
\SUMS{1}{b}{a} = \frac{(a-1)(b-1)}{2}
\end{equation*}

\section{Conclusion}

In this paper we have given closed-form formulas for $\mathbf{S}_2$, $\mathbf{S}_3$, $\mathbf{W}$ and the classical Dedekind sums. In addition we have shown the reciprocity laws that these expressions satisfy.  
 In the last section we have shown how $\mathbf{S}_r$ depend on generalized Dedekind sums through equation \ref{eq:Srba_rel_Dedekind_sums}. For $r\geq 4$ there is more than one Dedekind sum involved in the formula of $\mathbf{S}_r$ making the analysis more difficult than it is for $\mathbf{S}_2$ and $\mathbf{S}_3$ where only one Dedekind sum is involved, however that could probably be further investigated.

\section*{Acknowledgements} 

I would like to thank William Gasarch for kindly accepting to peer review this paper and for providing insightful comments. I also thank my wife Natallia for her continuous support.

\bibliographystyle{plain}
\bibliography{biblio}

\end{document}